\documentclass[11pt]{article}

\usepackage{amsmath,amsthm,amssymb,enumitem,dsfont}
\usepackage{color}

\usepackage[margin=1.3in]{geometry}
\usepackage[title]{appendix}

\usepackage{hyperref}
\newcommand{\R}{\mathbb{R}}

\newcommand{\ie}{{\it i.e.}}

\newcommand{\gap}{\mathsf{gap}}

\newtheorem{lem}{Lemma}
\newtheorem{thm}{Theorem}
\newtheorem{coro}{Corollary}

\newtheorem{prop}{Proposition}

\begin{document}

\title{Short-step Methods Are Not Strongly Polynomial-Time}

\date{ }

\author{Manru Zong\thanks{The Hong Kong Polytechnic University.  E--mail: {\tt manru.zong@connect.polyu.hk}} \and Yin Tat Lee\thanks{University of Washington.  E--mail: {\tt yintat@uw.edu}} \and Man-Chung Yue\thanks{The Hong Kong Polytechnic University.  E--mail: {\tt manchung.yue@polyu.edu.hk}}}

\maketitle

\begin{abstract}
Short-step methods are an important class of algorithms for solving convex constrained optimization problems. 
In this short paper, we show that under very mild assumptions on the self-concordant barrier and the width of the $\ell_2$-neighbourhood, any short-step interior-point method is not strongly polynomial-time.
\end{abstract}

\section{Introduction}\label{sec:intro}
An algorithm for solving linear programming problems is said to be strongly polynomial-time if the required number of arithmetic operations is a polynomial in the numbers of variables and constraints, independent of the bit-length for encoding the problem instance.
A major open problem in optimization and computer science is 
whether there exists a strongly polynomial-time algorithm for solving linear programming problems.


This paper is concerned with interior-point methods~\cite{nesterov1994interior}, which are an important class of algorithms for solving convex constrained optimization problems and enjoy great success in both theory and practice~\cite{gondzio2012interior, lee2021universal, nesterov1994interior, toh1999sdpt3}.
In particular, we focus on the sub-class of path-following interior-point methods~\cite{nesterov1994interior}. These methods are built upon the concept of self-concordant barriers~\cite{nesterov1994interior}. A self-concordant barrier allows us to define a curve, called the central path, leading from a certain center of the feasible region to an optimal solution. Roughly speaking, the principle of path-following interior-point methods is to approximately trace the central path to move towards the optimal solution. The geometry of the central path therefore plays an important role in the studies of path-following interior-point methods.

By constructing a family of linear programming problems with ill-behaved central paths, it has been shown in~\cite{deza2009central} that the number of iterations of any path-following interior-point method is lower bounded by $\Omega (\sqrt{m/(\log m)^3})$, where $m$ is the number of constraints.
Based on another family of linear programming problems, the recent paper~\cite{allamigeon2018log} improved the lower bound to $\Omega (2^m)$, see also~\cite{allamigeon2021tropical}. This lower bound applies to and thus denies the strong polynomiality of a large class of path-following interior-point methods, including the short-step methods associated with the logarithmic barrier~\cite{kojima1989polynomial, monteiro1989interior}, the long-step methods~\cite{kojima1989primal} and the predictor-corrector methods~\cite{mizuno1993adaptive}.
More instances of ill-behaved central paths are presented in~\cite{deza2008polytopes, gilbert2001examples}.



One way to generalize the short-step methods is to replace the logarithmic barrier by an arbitrary self-concordant barrier~\cite{nesterov1994interior}, such as the volumetric barrier~\cite{vaidya1989new}, universal barrier~\cite{lee2021universal, nesterov1994interior}, entropic barrier~\cite{bubeck2015entropic} and Lee-Sidford barrier~\cite{lee2019solving}.
Such an idea has led to improved algorithms for solving linear programming problems, see~\cite{lee2019solving, vaidya1989new} for example.

In view of the above-mentioned results, it is interesting to ask whether the short-step method associated with a general self-concordant barrier is strongly polynomial-time.  The purpose of this short paper is to settle this question negatively, see Corollary~\ref{coro:main}.
Our study is also motivated by the paper~\cite{allamigeon2020tropicalization}, which proved that for a family of linear programming problems, slightly different from the one in~\cite{allamigeon2018log}, a certain limit of the central path associated with the entropic barrier coincides with the central path associated with the logarithmic barrier. This suggests that the exponential lower bound in~\cite{allamigeon2018log} for short-step methods associated with the logarithmic barrier could possibly be generalized to short-step methods associated with a general self-concordant barrier. Our main result Corollary~\ref{coro:main} partly fills this gap.


Finally, we should also point out that some sub-classes of linear programming problems do admit strongly polynomial-time algorithms~\cite{adler1991strongly, tardos1986strongly, ye2011simplex}. 

\section{Preliminaries}
Consider the linear programming problem
\begin{equation}
\label{opt:LP_P} \tag{P}
\begin{array}{c@{\quad}l}
\mbox{minimize} & c^\top x \\
\noalign{\smallskip}
\mbox{subject to} & Ax \le b, 
\end{array}
\end{equation}
where $c\in\R^n$, $b\in \R^m$ and $A\in \R^{m\times n}$.
Denote the feasible region of problem~\eqref{opt:LP_P} and its interior by $\mathcal{F}$ and $\mathcal{F}^\circ$, respectively.
We assume that $\mathcal{F}$ is bounded and that $\mathcal{F}^\circ$ is non-empty. 
For simplicity, we denote the optimality gap of a feasible point $x\in\mathcal{F}$ by
\[ \mathsf{gap}(x) = c^\top x - \min_{x'\in\mathcal{F}} c^\top x' .\]


The notion of self-concordant barriers plays vital role in the study of interior-point methods.
Let $\mathcal{K}\subseteq\R^n$ be a proper convex domain, \ie, a convex set with non-empty interior and containing no 1-dimensional affine subspace. A function $\phi: \mathrm{int}(\mathcal{K})\to \R$ is said to be a barrier on $\mathcal{K}$ if 
$\phi(x) \to +\infty \quad \text{as}\quad x\to \partial \mathcal{K}$.
A three times continuously differentiable convex function $\phi$ is said to be self-concordant on $\mathcal{K}$ if for any $x\in \mathrm{int}(\mathcal{K})$ and $h\in \R^n$, 
\begin{equation*}\label{eq:SC}
\left| \mathrm{D}^3\phi(x) [h,h,h] \right| \le 2 \left( \mathrm{D}^2 \phi (x) [h,h] \right)^{\frac{3}{2}}.
\end{equation*}
If $\phi$ additionally satisfies that for any $x\in \mathrm{int}(\mathcal{K})$ and $h\in \R^n$, 
\begin{equation*}\label{eq:nu_SC}
\left| \mathrm{D} \phi(x) [h] \right| \le \left( \nu\, \mathrm{D}^2\phi(x) [h,h] \right)^{\frac{1}{2}}, 
\end{equation*}
then $\phi$ is said to be $\nu$-self-concordant on $\mathcal{K}$. This paper focuses on the proper convex domain $\mathcal{K} = \mathcal{F}$, which is a polytope defined by the linear inequalities $Ax \le b$. A standard self-concordant barrier on $\mathcal{F}$ is the logarithmic barrier
\begin{equation*}
\phi_{\ln} (x) = -\sum_{i=1}^m \ln (b - Ax)_i.
\end{equation*}

Given a self-concordant barrier $\phi$ on $\mathcal{F}$, we consider the problem 
\begin{equation}
\label{opt:center} \tag{P$_\mu$}
\begin{array}{c@{\quad}l}
\mbox{minimize} & c^\top x + \mu\, \phi(x) \\
\noalign{\smallskip}
\mbox{subject to} & Ax < b.
\end{array}
\end{equation}
From \cite[Section 5.3.4]{nesterov2018lectures}, for any $\mu >0$, there exists a unique minimizer $x^\phi (\mu)$ to problem~\eqref{opt:center}. We call $x^\phi (\mu)$ the $\mu$-analytic center of problem~\eqref{opt:LP_P} associated with the self-concordant barrier $\phi$. The central path of problem~\eqref{opt:LP_P} associated with $\phi$ is then defined as (the image of) the curve $\mu \mapsto  x^\phi(\mu)$ for $\mu>0$. By \cite[Theorem 5.3.10]{nesterov2018lectures}, $x^\phi(\mu)$ converges to an optimal solution to problem~\eqref{opt:LP_P} as $\mu \to 0$.  
 

By \cite[Theorem 5.1.6]{nesterov2018lectures}, $\nabla^2 \phi(x)$ is positive definite for any $x\in \mathcal{F}^\circ$. This allows us to define a norm
\begin{equation*}
\| h \|_x = \sqrt{ h^\top\nabla^2 \phi(x)h }, \quad h\in\R^n ,
\end{equation*}
and its dual norm
\begin{equation*}
\| h \|_x^* = \sqrt{ h^\top\left(\nabla^2 \phi(x)\right)^{-1}h }, \quad h\in\R^n .
\end{equation*}
For any $\theta \in (0,1)$ and $\mu>0$, the $\ell_2$-neighbourhood of problem~\eqref{opt:LP_P} is defined as 
\[\mathcal{N}^\phi_\theta (\mu) = \left\lbrace
x\in \mathcal{F}^\circ : \left\| c + \mu \, \nabla\phi (x) \right\|_x^* \le \theta \mu 
\right\rbrace .\]
Note that $c+ \mu \nabla \phi(x)$ is the gradient of the objective function of problem~\eqref{opt:center}. The $\ell_2$-neighbourhood is important and a natural choice for the design of path-following interior-point methods since the Newton's method applied to problem~\eqref{opt:center} converges quadratically to the optimal solution $x^\phi(\mu)$~\cite[Theorem~5.2.2]{nesterov2018lectures}.
We write
\begin{equation*}
\label{eq:ell_2_N}
\mathcal{N}^\phi_\theta = \bigcup_{\mu >0} \mathcal{N}^\phi_\theta (\mu),
\end{equation*}
which is also called the $\ell_2$-neighbourhood of problem~\eqref{opt:LP_P}. By the optimality condition of problem~\eqref{opt:center}, we can see that $x^\phi (\mu) \in \mathcal{N}^\phi_\theta (\mu)$ for any $\theta \in (0,1)$ and $\mu > 0$.
%

A short-step method associated with $\phi$ is defined as an algorithm that generates a sequence of iterates $\{x^k\}_{k\ge 0}$ such that the polygonal (\ie, continuous piecewise linear) curve formed using the sequence $\{x^k\}_{k\ge 0}$ is contained in the $\ell_2$-neighbourhood $\mathcal{N}^\phi_\theta$ of problem~\eqref{opt:LP_P} for some $\theta\in (0,1)$, where $k$ is the iteration counter, or more precisely,
\begin{equation*}
\bigcup_{k\ge 0} [x^k, x^{k+1}] \subseteq \mathcal{N}^\phi_\theta .
\end{equation*}

Another popular choice of the neighbourhood for path-following algorithms for solving linear programming problems is the so-called wide neighbourhood~\cite{wright1997primal}:
\begin{align*}
\mathcal{W}_\theta (\mu) = \Big\lbrace x\in \mathcal{F}^\circ :&\, \exists\, y\in\R_+^m \text{ such that } A^\top y = -c,\ y^\top (b - A x) =m \mu , \\
&\, \text{ and }y_i\, (b - Ax)_i \ge (1-\theta) \mu,\  i = 1,\dots, m\Big\rbrace.
\end{align*}
Similarly to the $\ell_2$-neighbourhood, we write
\begin{equation*}
\mathcal{W}_\theta = \bigcup_{\mu > 0} \mathcal{W}_\theta (\mu).
\end{equation*}

\section{Main Results}
\label{sec:main}
Our non-strong polynomiality result is based on the following family of linear programs introduced in~\cite{allamigeon2018log}:
\begin{equation} \label{opt:LW} \tag*{$\normalfont{\textbf{LW}}_{r}(t)$}
\begin{array}{c@{\quad}l}
\mbox{minimize} & x_1 \\
\noalign{\smallskip}
\mbox{subject to} & x_1 \le t^2, \\
& x_2 \le t,\\
& x_{2j + 1} \le t\, x_{2j - 1}, \quad j = 1,\dots, r-1,\\
& x_{2j+1} \le t\, x_{2j}, \quad j = 1,\dots, r-1,\\
& x_{2j+2} \le t^{1 - 1/2^j} (x_{2j-1} + x_{2j}), \quad j = 1,\dots, r-1,\\
& x_{2r-1},\, x_{2r} \ge 0,
\end{array}
\end{equation}
where $t>1$ is a real number and $r\ge 1$ is an integer.
The notation \ref{opt:LW} follows from~\cite{allamigeon2018log} and signifies that the central path of this linear program is long and winding. 

To distinguish the properties of problem~\ref{opt:LW} from those of a general linear program, we introduce an extra subscript $t$ to the notations. For instances, the feasible region, the $\mu$-analytic center associated with a self-concordant barrier $\phi$ and the wide neighbourhood of problem~\ref{opt:LW} are denoted as $\mathcal{F}_t$, $ x^\phi_t (\mu) $ and $ \mathcal{W}_{\theta , t} $, respectively.

We can now present the main results of this paper, whose proofs are deferred to Section~\ref{sec:proofs}.
The main contribution of this paper is the non-strong polynomiality of short-step methods.
\begin{coro}[Non-strong Polynomiality of Short-step Methods]
\label{coro:main}
Consider \ref{opt:LW} for a sufficiently large $t>1$.
Let $\theta\in (0, \tfrac{1}{2})$, $\phi$ be a $\nu$-self-concordant barrier on $\mathcal{F}_t$ with $\nu$ independent of $t$ and $\{x^k\}_{k\ge 0}$ be a sequence of iterates generated by a short-step method associated with $\phi$. 
Suppose that 
\[ \gap(x^0) \ge 320r\nu \sqrt{t}   \quad\text{and}\quad  \gap(x^K) \le  \frac{1}{180r}  .\]
Then, $K \ge 2^{r-3}$.
\end{coro}
\noindent We should emphasize that although the self-concordance parameter $\nu$ in Corollary~\ref{coro:main} is assumed to be independent of $t$, it could possibly depend on the numbers of variables and constraints of problem~\ref{opt:LW}. This subsumes almost all known barriers. Furthermore, the constants in Corollary~\ref{coro:main} are simplified for readability. They are not optimal and can be made slightly better (see Corollary~\ref{coro:main_precise}). Finally, we note that the requirement on the initial optimality gap is only very mild. Indeed, it is customary for interior-point methods to start at the $\infty$-analytic center $x^\phi(\infty)$. Using Theorem~\ref{thm:gap}, we can check that $\gap(x^\phi(\infty)) = \Omega(\tfrac{t}{\nu})$, see also the proof of Lemma~\ref{lem:eta_mu}.

The cruxes of the proof of Corollary~\ref{coro:main} are the following theorems.
The first shows a certain equivalence among the central paths of all self-concordant barriers. Geometrically speaking, it guarantees that when restricted to a constant-cost slice, all central paths are approximately equally close to (or away from) the boundary of the feasible region.
\begin{thm}[Equivalence of Central Paths]
\label{thm:asymp_equivalence_center}
Consider the linear program~\eqref{opt:LP_P}. Let $\phi$ and $\psi$ be $\nu_\phi$- and $\nu_\psi$-self-concordant barriers on the feasible region $\mathcal{F}$.
Let $\mu, \eta > 0$  be such that $c^\top x^\phi(\mu) = c^\top x^\psi (\eta)$.  Then, for any $i=1,\dots,m$,
\begin{equation*}
(1+ \nu_\phi + 2\sqrt{\nu_\phi})^{-1} \le \frac{(b - Ax^\phi(\mu))_i}{(b - Ax^\psi(\eta))_i} \le (1+ \nu_\psi + 2\sqrt{\nu_\psi}).
\end{equation*}
\end{thm}
The second one bounds the optimality gap of analytic centers. 

\begin{thm}[Optimality Gap of Analytic Center]
\label{thm:gap}
Consider the linear program~\eqref{opt:LP_P}. Let $\phi$ be a $\nu$-self-concordant barrier on the feasible region $\mathcal{F}$. Then, for any $\mu > 0$,
\begin{equation*}
\min\left\lbrace \frac{\mu}{2}, \frac{ \rho \|c\|_2}{2\nu + 4\sqrt{\nu}} \right\rbrace \le \gap(x^\phi(\mu)) \le \mu \nu,
\end{equation*}
where $\rho > 0$ is the radius of the largest ball contained in $\mathcal{F}$.
\end{thm}
\noindent It should be pointed out that the upper bound in Theorem~\ref{thm:gap} is not new and can be found in~\cite[Theorem~5.3.10]{nesterov2018lectures}. The novelty of the theorem lies in the lower bound.

\section{Proofs}\label{sec:proofs}
In this section, we prove the main results.
Throughout this section, given a vector $\bar{x}\in \R^n$ and a positive definite matrix $Q$, we let $\mathcal{E}(Q,\bar{x}) = \{ x\in\R^n: (x-\bar{x})^\top Q(x-\bar{x}) \le1 \}$. Also, for any $\nu > 0$, we let $C_\nu = \nu + 2\sqrt{\nu}$. 

\subsection{Proof of Theorem~\ref{thm:asymp_equivalence_center}}
We first collect a simple lemma about linear optimization over ellipsoids, whose proof uses only elementary optimality arguments and is thus omitted.
\begin{lem}\label{lem:ellipsoid}
Let $\bar{x}\in\R^n$, $a\in\R^n$ and $Q\in\R^{n\times n}$ be a positive definite matrix. Then,
\begin{equation*}
\max_{ x \in \mathcal{E}(Q, \bar{x}) } a^\top x = a^\top \bar{x} + \sqrt{a^\top Q^{-1} a}\quad\text{and}\quad \min_{x \in \mathcal{E}(Q, \bar{x})} a^\top x = a^\top \bar{x} - \sqrt{a^\top Q^{-1} a}.
\end{equation*}
\end{lem}

We can now prove Theorem~\ref{thm:asymp_equivalence_center}.
\begin{proof}[Proof of Theorem~\ref{thm:asymp_equivalence_center}]
Let $\mathcal{H} = \{x\in \R^n: c^\top x = c^\top x^\phi(\mu) \} $. 
By \cite[Theorem~5.1.5]{nesterov2018lectures}, we have
\begin{equation*}\label{eq:dikin}
\mathcal{E} ( \nabla^2 \phi(x^\phi(\mu)), x^\phi(\mu))  \subseteq \mathcal{F} \quad\text{and}\quad
\mathcal{E} ( \nabla^2 \psi(x^\psi(\eta)), x^\psi(\eta))  \subseteq \mathcal{F}.
\end{equation*}
Combining~\cite[Theorem~5.3.8]{nesterov2018lectures} with these inclusions yields
\begin{align*}
\mathcal{E} ( \nabla^2 \phi(x^\phi(\mu)), x^\phi(\mu)) \cap \mathcal{H} \subseteq \mathcal{F}\cap \mathcal{H} \subseteq \mathcal{E} ( C_{\nu_\phi}^{-2} \nabla^2 \phi(x^\phi(\mu)), x^\phi(\mu)) \cap \mathcal{H} ,
\end{align*}
and 
\begin{align*}
\mathcal{E} ( \nabla^2 \psi(x^\psi(\eta)), x^\psi(\eta)) \cap \mathcal{H} \subseteq \mathcal{F}\cap \mathcal{H} \subseteq \mathcal{E} ( C_{\nu_\psi}^{-2} \nabla^2 \psi(x^\psi(\eta)), x^\psi(\eta)) \cap \mathcal{H} ,
\end{align*}
which implies, respectively,
\begin{equation*}
x^\psi(\eta) \in \mathcal{E} ( C_{\nu_\phi}^{-2} \nabla^2 \phi(x^\phi(\mu)), x^\phi(\mu)) \cap \mathcal{H} \quad\text{and}\quad x^\phi(\mu) \in \mathcal{E} ( C_{\nu_\psi}^{-2} \nabla^2 \psi(x^\psi(\eta)), x^\psi(\eta)) \cap \mathcal{H}.
\end{equation*}
Using Lemma~\ref{lem:ellipsoid}, for any $i=1,\dots,m$,
\begin{equation}\label{eq:max}
\max_{x'\in \mathcal{E}(C_{\nu_\phi}^{-2} \nabla^2 \phi (x^\phi(\mu)), x^\phi (\mu))} (b - Ax')_i = (b - Ax^\phi (\mu))_i + C_{\nu_\phi} \sqrt{a_i^\top  \left(\nabla^2 \phi(x^\phi(\mu)) \right)^{-1} a_i},
\end{equation}
and
\begin{equation}\label{eq:min}
\min_{x'\in \mathcal{E}(C_{\nu_\phi}^{-2} \nabla^2 \phi (x^\phi(\mu)), x^\phi (\mu))} (b - Ax')_i = (b - Ax^\phi (\mu))_i - C_{\nu_\phi} \sqrt{a_i^\top  \left(\nabla^2 \phi(x^\phi(\mu)) \right)^{-1} a_i} ,
\end{equation}
where $a_i^\top$ is the $i$-row of $A$.
Also, since $Ax' \le b$ for any $x' \in \mathcal{E}( \nabla^2 \phi (x^\phi(\mu)), x^\phi (\mu)) \subseteq \mathcal{F}$, Lemma~\ref{lem:ellipsoid} implies that
\begin{equation*}
0\le  \min_{x'\in \mathcal{E}( \nabla^2 \phi (x^\phi(\mu)), x^\phi (\mu))} (b - Ax')_i = (b - Ax^\phi (\mu))_i -  \sqrt{a_i^\top  \left(\nabla^2 \phi(x^\phi(\mu)) \right)^{-1} a_i} .
\end{equation*}
It follows from \eqref{eq:max}, \eqref{eq:min} and $ x^\psi(\eta) \in  \mathcal{E} ( C_{\nu_\phi}^{-2} \nabla^2 \phi(x^\phi(\mu)), x^\phi(\mu))$ that
\begin{align*}
& \left| (b - Ax^\phi(\mu))_i - (b - Ax^\psi(\eta))_i \right| 
\le  C_{\nu_\phi} \sqrt{a_i^\top \left(\nabla^2 \phi(x^\phi(\mu))\right)^{-1} a_i} \le C_{\nu_\phi} (b - Ax^\phi(\mu))_i.
\end{align*}
Using the same arguments, we also obtain
\begin{align*}
& \left| (b - Ax^\phi(\mu))_i - (b - Ax^\psi(\eta))_i \right| 
\le  C_{\nu_\psi} \sqrt{a_i^\top \left(\nabla^2 \psi(x^\psi(\eta))\right)^{-1} a_i} \le C_{\nu_\psi} (b - Ax^\psi(\eta))_i.
\end{align*}
Thus, for any $i=1,\dots,m$,
\begin{align*}
(b - Ax^\psi(\eta))_i \le (1+C_{\nu_\phi}) (b - Ax^\phi(\mu))_i  \quad\text{and}\quad (b - Ax^\phi(\mu))_i \le (1+C_{\nu_\psi}) (b - Ax^\psi(\eta))_i ,
\end{align*}
which completes the proof.
\end{proof}

\subsection{Proof of Theorem~\ref{thm:gap}}
We next prove Theorem~\ref{thm:gap}.
\begin{proof}[Proof of Theorem~\ref{thm:gap}]
The upper bound follows directly from~\cite[Theorem~5.3.10]{nesterov2018lectures}. Therefore, we prove only the lower bound. Two cases are discussed separately: 
\begin{align*}
&\, \nabla \phi(x^\phi(\mu)))^\top \left(\nabla^2 \phi(x^\phi(\mu))\right)^{-1} \nabla \phi(x^\phi(\mu))) \le \frac{1}{4} \\
\quad\text{and}\quad &\, \nabla \phi(x^\phi(\mu)))^\top \left(\nabla^2 \phi(x^\phi(\mu))\right)^{-1} \nabla \phi(x^\phi(\mu))) > \frac{1}{4}.
\end{align*}
We first consider the case of 
\[ \nabla \phi(x^\phi(\mu)))^\top \left(\nabla^2 \phi(x^\phi(\mu))\right)^{-1} \nabla \phi(x^\phi(\mu))) \le \frac{1}{4} . \]
By \cite[Lemma~5.1.5 and Theorem~5.2.1]{nesterov2018lectures}, we have
\begin{align*}
\frac{\|  x^\phi (\mu) - x^\phi (\infty) \|_{x^\phi (\infty)}^2}{1+ \frac{2}{3}\| x^\phi (\mu) - x^\phi (\infty) \|_{x^\phi (\infty)}} \le \frac{(\tfrac{1}{4})^2}{1 - \tfrac{1}{4}} = \frac{1}{12},
\end{align*}
where $x^\phi (\infty)$ is the $\infty$-analytic center, \ie, the unique optimal solution to the problem
\begin{equation*}
\begin{array}{c@{\quad}l}
\mbox{minimize} &  \phi(x) \\
\noalign{\smallskip}
\mbox{subject to} & Ax < b.
\end{array}
\end{equation*}
Solving the quadratic inequality, we get
\begin{equation}
\label{ineq:thm2_proof_1}
\| x^\phi(\mu) - x^\phi (\infty) \|_{x^\phi(\infty)} \le \frac{1}{2}.
\end{equation}
Using Lemma~\ref{lem:ellipsoid} and inequality~\eqref{ineq:thm2_proof_1}, we have that 
\begin{equation}
\label{ineq:thm2_proof_2}
\begin{split}
\gap(x^\phi(\mu)) & \ge \min_{x\in \mathcal{E}( 16 \nabla^2 \phi(x^\phi(\infty)) , x^\phi(\infty))} \gap (x) \\
& = \gap(x^\phi(\infty)) - \frac{1}{2}\sqrt{c^\top \left( \nabla^2 \phi(x^\phi(\infty)) \right)^{-1} c} .
\end{split}
\end{equation}
Next, since $\mathcal{E}( \nabla^2 \phi(x^\phi(\infty)) , x^\phi(\infty)) \subseteq \mathcal{F}$, Lemma~\ref{lem:ellipsoid} implies
\begin{equation}
\label{ineq:thm2_proof_4}
0 \le \min_{x\in \mathcal{E}( \nabla^2 \phi(x^\phi(\infty)) , x^\phi(\infty))} \gap (x) \\
 = \gap(x^\phi(\infty)) - \sqrt{c^\top \left( \nabla^2 \phi(x^\phi(\infty)) \right)^{-1} c}.
\end{equation}
Combining the inequalities~\eqref{ineq:thm2_proof_2} and~\eqref{ineq:thm2_proof_4}, we obtain
\begin{equation}
\label{ineq:thm2_proof_3}
\gap(x^\phi(\mu)) \ge \frac{1}{2}\sqrt{c^\top \left( \nabla^2 \phi(x^\phi(\infty)) \right)^{-1} c}.
\end{equation}
On the other hand, by supposition, the feasible region $\mathcal{F}$ contains a ball of radius $\rho$. By \cite[Theorem~5.3.9]{nesterov2018lectures}, it is contained in the ellipsoid $\mathcal{E}(C_\nu^{-2} \nabla^2\phi (x^\phi(\infty)), x^\phi(\infty) )$. Hence,
\[ \mathcal{E}(\rho^{-2} I, x^\phi(\infty)) \subseteq \mathcal{E}(C_\nu^{-2} \nabla^2\phi (x^\phi(\infty)), x^\phi(\infty) ) ,\]
where the left-hand side is a ball with center $x^\phi(\infty)$ of radius $\rho$.
Using Lemma~\ref{lem:ellipsoid}, we get 
\begin{align*}
& \,\gap( x^\phi(\infty)) + \rho \|c\|_2  = \max_{x \in \mathcal{E}(\rho^{-2} I, x^\phi(\infty))} \gap(x) 
\le \max_{x \in \mathcal{E}(C_\nu^{-2} \nabla^2\phi (x^\phi(\infty)), x^\phi(\infty) )} \gap(x) \\
= &\, \gap( x^\phi(\infty)) + (\nu+2\sqrt{\nu} )  \sqrt{c^\top \left( \nabla^2 \phi(x^\phi(\infty)) \right)^{-1} c},
\end{align*}
which, upon substitution into~\eqref{ineq:thm2_proof_3}, yields
\[ \gap(x^\phi(\mu)) \ge \frac{\rho \|c\|_2}{2(\nu+2\sqrt{\nu})} .\]
Then, we consider the case of 
\[ \nabla \phi(x^\phi(\mu)))^\top \nabla^2 \phi(x^\phi(\mu)) \nabla \phi(x^\phi(\mu))) > \frac{1}{4} .\]
Using the optimality condition of problem~\eqref{opt:center}, we get $c + \mu \nabla \phi(x^\phi(\mu)) = 0$.
Therefore,
\begin{equation}
\label{ineq:thm2_proof_5}
c^\top \left(\nabla^2 \phi(x^\phi(\mu))\right)^{-1} c > \frac{1}{4}\mu^2. 
\end{equation}
Since $ \mathcal{E}( \nabla^2 \phi(x^\phi(\mu)), x^\phi(\mu) ) \subseteq \mathcal{F}$, by Lemma~\ref{lem:ellipsoid}, 
\[ 0 \le \min_{x\in \mathcal{E}( \nabla^2 \phi(x^\phi(\mu)), x^\phi(\mu) ) } = \gap(x^\phi(\mu)) - \sqrt{ c^\top \left(\nabla^2 \phi(x^\phi(\mu))\right)^{-1} c }, \]
which, together with \eqref{ineq:thm2_proof_5}, yields
\[ \gap(x^\phi(\mu)) > \frac{\mu}{2}. \]
This completes the proof.
\end{proof}

\subsection{Some Extra Tools}
The following proposition compares the slack of the $\mu$-analytic center with that of any point in its $\ell_2$-neighbourhood.
\begin{prop}\label{prop:distance_to_center}
Consider the linear program~\eqref{opt:LP_P}. Let $\theta\in (0, (\sqrt{69}-3)/10)$, $\mu >0$, $\phi$ be a self-concordant function on $\mathcal{F}$ and $x\in\mathcal{N}^\phi_\theta (\mu)$.  
Then, for any $i=1,\dots,m$,
\begin{equation*}
(1-\beta_\theta) (b - A x^\phi (\mu))_i \le ( b - Ax )_i \le (1+\beta_\theta) (b - A x^\phi (\mu))_i,
\end{equation*}
where $\beta_\theta \in (0,1)$ is a constant depending only on $\theta$ defined in~\eqref{eq:beta}.
\end{prop}
\begin{proof}
Let $x\in\mathcal{N}^\phi_\theta (\mu)$.
By \cite[Lemma~5.1.5 and Theorem~5.2.1]{nesterov2018lectures} and the definition of $\ell_2$-neighbourhood,
\begin{align*}
\frac{\| x - x^\phi (\mu) \|_{x^\phi(\mu) }^2}{1+ \frac{2}{3}\| x - x^\phi (\mu) \|_{x^\phi(\mu) }} \le \frac{\theta^2}{1 -\theta}.
\end{align*}
Solving the quadratic inequality yields $\| x - x^\phi (\mu) \|_{x^\phi(\mu) } \le \beta_\theta$, where
\begin{equation}\label{eq:beta}
\beta_\theta = \frac{1}{3}\left(\frac{\theta^2}{1 -\theta} + \sqrt{\frac{\theta^4}{(1 -\theta)^2} + \frac{9\theta^2}{1 -\theta}}\right).
\end{equation}
We note that $\beta_\theta < 1$ whenever $\theta\in (0, (\sqrt{69}-3)/10)$.
Using Lemma~\ref{lem:ellipsoid}, we get
\begin{equation*}
\max_{x'\in \mathcal{E}(\beta_\theta^{-2} \nabla^2 \phi (x^\phi(\mu)), x^\phi (\mu))} (b - Ax')_i = (b - Ax^\phi (\mu))_i + \beta_\theta \sqrt{a_i^\top  \left(\nabla^2 \phi(x^\phi(\mu)) \right)^{-1} a_i},
\end{equation*}
and
\begin{equation*}
\min_{x'\in \mathcal{E}(\beta_\theta^{-2} \nabla^2 \phi (x^\phi(\mu)), x^\phi (\mu))} (b - Ax')_i = (b - Ax^\phi (\mu))_i - \beta_\theta \sqrt{a_i^\top  \left(\nabla^2 \phi(x^\phi(\mu)) \right)^{-1} a_i} \ge 0,
\end{equation*}
where the  inequality follows from that $Ax' \le b$ for any $x' \in \mathcal{E}(\beta_\theta^{-2} \nabla^2 \phi (x^\phi(\mu)), x^\phi (\mu)) \subseteq \mathcal{F}$.
Therefore, we have 
\begin{equation*}
\left| (b - Ax)_i - (b - Ax^\phi (\mu))_i \right| \le \beta_\theta \sqrt{a_i^\top  \left(\nabla^2 \phi(x^\phi(\mu)) \right)^{-1} a_i} \le \beta_\theta   (b - Ax^\phi (\mu))_i.
\end{equation*}
This completes the proof.
\end{proof}

The next proposition compares the optimality gap of the $\mu$-analytic center with that of any point in its $\ell_2$-neighbourhood.
\begin{prop}
\label{prop:gap_ratio}
Consider the linear program~\eqref{opt:LP_P}. Let $\theta\in (0, (\sqrt{69}-3)/10)$, $\mu >0$, $\phi$ be a self-concordant function on $\mathcal{F}$ and $x\in\mathcal{N}^\phi_\theta (\mu)$.  
Then,
\begin{equation*}
(1-\beta_\theta) \gap(x^\phi(\mu)) \le \gap(x) \le (1+\beta_\theta) \gap(x^\phi(\mu)).
\end{equation*}
\end{prop}
\noindent
The proof of Proposition~\ref{prop:gap_ratio} uses the same arguments as in the proof of Proposition~\ref{prop:distance_to_center} and is therefore omitted.

We will need to compare the $\mu$-analytic center associated with $\phi$ and the $\eta$-analytic center associated with the logarithmic barrier $\phi_{\ln}$ along the constant-cost slices of $\mathcal{F}_t$.

\begin{lem}\label{lem:eta_mu}
Consider the linear program~\ref{opt:LW}. For any self-concordant barrier $\phi$ on $\mathcal{F}_t$ and $\mu > 0$ satisfying 
\[ \gap(x^\phi(\mu)) < \frac{t}{4(3r+1) + 8\sqrt{3r+1}}, \]
there exists a unique $\eta(\mu) >0$ such that
\[  \frac{\eta(\mu)}{2}\le 
\gap(x^\phi(\mu)) = \gap(x^{\phi_{\ln}} (\eta(\mu)) ) \le (3r+1) \eta(\mu )  . \]
\end{lem}
\begin{proof}
We first prove the existence of $\eta(\mu)$ satisfying $\gap(x^\phi(\mu)) = \gap(x^{\phi_{\ln}} (\eta(\mu)))$. Since $\mu\mapsto\gap(x^\phi(\mu))$ and $\eta\mapsto\gap(x^{\phi_{\ln}} (\eta(\mu)))$ are both continuous and strictly increasing functions attaining the minimum value $0$ at $0$, it suffices to show that for any $\mu > 0$ with 
\[ \gap(x^\phi(\mu)) < \frac{t}{4(3r+1) + 8\sqrt{3r+1}}, \]
there exists $\eta > 0$ such that $ \gap(x^\phi(\mu)) \le \gap(x^{\phi_{\ln}} (\eta) )$. To show this, we note that the feasible region $\mathcal{F}_t$ contains a hypercube of side length $t$, which in turn contains a ball of radius $\frac{t}{2}$. Also, the cost vector of problem~\ref{opt:LW} is $c_t = (1,0,\dots,0)^\top$ and hence $\|c_t\|_2 = 1$. Furthermore, it is well-known that the self-concordance parameter of the logarithmic barrier $\phi_{\ln}$ on a polytope is precisely the number of linear inequalities defining the polytope, see~\cite[Section~5.3]{nesterov2018lectures} for example. Therefore, $\phi_{\ln}$ is $(3r+1)$-self-concordant on $\mathcal{F}_t$.
Theorem~\ref{thm:gap} then implies that for a large enough $\eta>0$,
\begin{align*}
\gap(x^\phi(\mu)) <  \min\left\lbrace\frac{t}{4(3r+1) + 8\sqrt{3r+1}}, \frac{\eta}{2} \right\rbrace \le \gap(x^{\phi_{\ln}} (\eta)).
\end{align*}
Therefore, there exists a unique $\eta(\mu) > 0$ with $\gap(x^\phi(\mu)) = \gap(x^{\phi_{\ln}} (\eta(\mu)))$.

It remains to prove the inequalities for $\gap(x^\phi(\mu)) = \gap(x^{\phi_{\ln}} (\eta(\mu)))$. By Theorem~\ref{thm:gap},
\begin{equation*}
\min\left\lbrace \frac{t}{4(3r+1) + 8\sqrt{3r+1}}, \frac{\eta(\mu)}{2} \right\rbrace \le 
\gap(x^{\phi_{\ln}} (\eta(\mu))) \le \eta(\mu) (3r+1) .
\end{equation*}
Since 
\[\gap(x^\phi(\mu)) = \gap(x^{\phi_{\ln}} (\eta(\mu))) < \frac{t}{4(3r+1) + 8\sqrt{3r+1}}, \]
we get 
\begin{equation*}
 \frac{\eta(\mu)}{2}  \le 
\gap(x^\phi(\mu)) = \gap(x^{\phi_{\ln}} (\eta(\mu))) \le \eta(\mu) (3r+1) .
\end{equation*}
This completes the proof.
\end{proof}

We will also need the following remarkable result from~\cite[Theorem~29]{allamigeon2018log}.
\begin{thm}\label{thm:wide_neighbourhood}
Let $\omega\in (0,1)$. Consider the linear program \ref{opt:LW} for 
\begin{equation*}
t > \left( \frac{2 (5r-1) (10 r - 1)^4 ((10r-2)!)^8}{1-\omega} \right)^{2^{r+2}} .
\end{equation*}
Suppose that
\begin{equation*}
[x^0, x^1] \cup [x^1,x^2]\cup\cdots \cup[x^{K-1},x^K] \subseteq \mathcal{W}_{\omega, t}, 
\end{equation*} 
with $x^0\in \mathcal{W}_{\omega,t} (\eta_0)$ and $x^K \in \mathcal{W}_{\omega,t} (\eta_K)$ for some $\eta_0 \ge \sqrt{t}$ and $\eta_K \le 1$. Then, $K \ge 2^{r-3}$.
\end{thm}
Some remarks are in order. First, the original statement of~\cite[Theorem~29]{allamigeon2018log} concerns iterates of a longer vector consisting of primal, dual and slack variables. But Theorem~\ref{thm:wide_neighbourhood} concerns only the primal variables and hence is seemingly stronger than~\cite[Theorem~29]{allamigeon2018log}, because, for a polygonal curve in $\R^N$ with $K$ pieces, its projection to a lower dimensional space could possibly have less than $K$ pieces in general. However, upon a close inspection of the proof of \cite[Theorem~29]{allamigeon2018log} and \cite[Section 6.2]{allamigeon2018log}, we find that Theorem~\ref{thm:wide_neighbourhood} holds for the linear program~\ref{opt:LW}.
Using the language of \cite{allamigeon2018log}, the reason is that for the tropical central path of the linear program~\ref{opt:LW}, even the projection onto the primal variables is already a polygonal curve with a huge number of pieces.
Second, in the original statement of~\cite[Theorem~29]{allamigeon2018log}, the requirement on $\eta_0$ is that $\eta_0 \ge t^2$, and the conclusion is that $K \ge 2^{r-1}$. By considering a shorter part of the tropical central path of \ref{opt:LW}, we find that replacing the requirement on $\eta_0$ with the weaker bound of $\eta_0 \ge \sqrt{t}$ leads the weaker conclusion of $K\ge 2^{r-3}$, see~\cite[Section~4.3]{allamigeon2018log}.


\subsection{Proof of Corollary~\ref{coro:main}}
We now have enough tools at our disposal to prove Corollary~\ref{coro:main}. In fact, we will prove the following version with slightly better constants, from which Corollary~\ref{coro:main} follows. 
\begin{coro}[Precise Version of Corollary~\ref{coro:main}]
\label{coro:main_precise}
Consider \ref{opt:LW} for a sufficiently large $t>1$.
Let $\theta\in (0, (\sqrt{69}-3)/10)$, $\phi$ be a $\nu$-self-concordant barrier on $\mathcal{F}_t$ with $\nu$ independent of $t$ and $\{x^k\}_{k\ge 0}$ be a sequence of iterates generated by a short-step method associated with $\phi$. 
Suppose that 
\[ \gap(x^0) \ge \frac{(1+\beta_\theta)(3r+1)(1+C_\nu)\sqrt{t}}{1-\beta_\theta}  \quad\text{and}\quad  \gap(x^K) \le  \frac{1-\beta_\theta}{2 (1+\beta_\theta) (1+C_{3r+1})}  ,\]
where $\beta_\theta \in (0,1)$ is a constant depending only on $\theta$ defined in~\eqref{eq:beta}.
Then, $K \ge 2^{r-3}$.
\end{coro}
\begin{proof}
Let
\[ \omega = 1- \frac{(1-\beta_\theta)}{(1+\beta_\theta) (1 + C_\nu )(1+C_{3r+1})}   \]
and 
\[ t >  \left( \frac{2 (5r-1) (10 r - 1)^4 ((10r-2)!)^8}{1-\omega} \right)^{2^{r+2}}. \]
Note that $\omega \in (0,1)$.

\subsubsection*{Reduction}

Without loss of generality, we can assume that for any $x\in [x^0,x^1] \cup \cdots\cup [x^{K-1} , x^K]$,
\begin{equation}
\label{ineq:gap_assume}
\gap(x^0) = \frac{(1+\beta_\theta)(3r+1)(1+C_\nu)\sqrt{t}}{1-\beta_\theta} \ge \gap(x) = \frac{1-\beta_\theta}{2 (1+\beta_\theta) (1+C_{3r+1})} = \gap(x^K).
\end{equation}
Indeed, if this is not the case, because of the continuity of $\gap$ and the assumption that 
\[\gap(x^0) \ge \frac{(1+\beta_\theta)(3r+1)(1+C_\nu)\sqrt{t}}{1-\beta_\theta} \quad \text{and}\quad \gap(x^K)\le \frac{1-\beta_\theta}{2 (1+\beta_\theta) (1+C_{3r+1})} ,\] we can choose a sub-curve (connected subset) of $[x^0,x^1] \cup \cdots\cup [x^{K-1} , x^K]$ that satisfies all the assumptions of Corollary~\ref{coro:main_precise} as well as assumption~\eqref{ineq:gap_assume}.

Our goal is to show that
$[x^0,x^1]\cup\cdots \cup [x^{K-1}, x^K] \subseteq \mathcal{W}_{\omega, t}  $
and that $x^0\in\mathcal{W}_{\omega, t} (\eta_0)$ and $x^K\in\mathcal{W}_{\omega, t} (\eta_K)$, for some $\eta_0 \ge \sqrt{t}$ and $\eta_K \le 1$.
If this can be proved, the desired conclusion would then follow from Theorem~\ref{thm:wide_neighbourhood}.

\subsubsection*{Proving $[x^0,x^1]\cup\cdots \cup [x^{K-1}, x^K] \subseteq \mathcal{W}_{\omega, t}  $}
Let $x\in [x^0,x^1]\cup\cdots \cup [x^{K-1}, x^K]$. Then, $x\in \mathcal{N}^\phi_{\theta, t} (\mu)$ for some $\mu >0$. By assumption~\eqref{ineq:gap_assume} and Proposition~\ref{prop:gap_ratio}, we have
\[\gap(x^\phi(\mu)) \le \frac{(1+\beta_\theta)^2 (3r+1)(1+C_\nu)\sqrt{t}}{1-\beta_\theta} < \frac{t}{4(3r+1) + 8\sqrt{3r+1}}.\] 
Using Lemma~\ref{lem:eta_mu}, there exists a unique $\eta({\mu}) > 0$ such that 
\[ c_t^\top x_t^\phi(\mu) = c_t^\top x_t^{\phi_{\ln}} (\eta({\mu} )) .\]
By Proposition~\ref{prop:distance_to_center} and Theorem~\ref{thm:asymp_equivalence_center}, for any $i=1,\dots,3r+1$,
\begin{equation}\label{eq:coro_proof_1}
\begin{split}
\frac{(b_t - A_t x)_i}{(b_t - A_t x^{\phi_{\ln}}_t (\eta(\mu)))_i} = \frac{(b_t - A_t x)_i}{(b_t - A_t x^{\phi}_t (\mu))_i}\cdot \frac{(b_t - A_t x^{\phi}_t (\mu))_i}{(b_t - A_t x^{\phi_{\ln}}_t (\eta(\mu)))_i} \ge \frac{(1-\beta_\theta)}{(1 + C_\nu )}   ,
\end{split}
\end{equation}
From the definition of the logarithmic barrier $\phi_{\ln}$ and the optimality conditions of problem~\eqref{opt:center}, there exists $y\in \R_+^{3r+1}$ such that $ A_t^\top y = - c_t$ and 
\[ y_i (b_t - A_tx^{\phi_{\ln}}_t(\eta(\mu)))_i = \eta(\mu), \quad i = 1,\dots, 3r+1.\]
Averaging these $3r+1$ equalities and then using Proposition~\ref{prop:distance_to_center}, Theorem~\ref{thm:asymp_equivalence_center} and the fact that $\phi_{\ln}$ is $(3r+1)$-self-concordant on $\mathcal{F}_t$, we get
\begin{equation}\label{eq:coro_proof_2}
\begin{split}
\eta (\mu) &= \frac{1}{3r+1} \sum_{i=1}^{3r+1} y_i (b_t - A_tx^{\phi_{\ln}}_t(\eta(\mu)))_i \\
& \ge \frac{(1 + C_{3r+1})^{-1}}{3r+1} \sum_{i=1}^{3r+1} y_i (b_t - A_tx^\phi_t(\mu))_i \\
& \ge (1 + C_{3r+1})^{-1} (1+\beta_\theta)^{-1} \frac{y^\top (b_t - A_t x )}{3r + 1}.   
\end{split}
\end{equation}
Hence, using~\eqref{eq:coro_proof_1} and~\eqref{eq:coro_proof_2}, for any $i=1,\dots,3r+1$,
\begin{align}
y_i (b_t - A_t x)_i & \ge \frac{(1-\beta_\theta)}{(1 + C_\nu )} \cdot y_i (b_t - A_tx^{\phi_{\ln}}_t(\eta(\mu))_i  = \frac{(1-\beta_\theta)}{(1 + C_\nu )}\cdot \eta (\mu)\label{eq:coro_proof_3} \\
&\ge \frac{(1-\beta_\theta)}{(1+\beta_\theta) (1 + C_\nu )(1+C_{3r+1})}\cdot \frac{y^\top (b_t - A_t x )}{3r + 1},\notag
\end{align}
which implies that $x\in \mathcal{W}_{\omega, t} (\eta_x)$ with 
\begin{equation}
\label{eq:coro_proof_4}
\eta_x = \frac{y^\top (b_t - A_t x )}{3r + 1} .
\end{equation}

\subsubsection*{Proving $x^0\in\mathcal{W}_{\omega, t} (\eta_0)$ and $x^K\in\mathcal{W}_{\omega, t} (\eta_K)$ for some $\eta_0 \ge \sqrt{t}$ and $\eta_K \le 1$}
Using definition~\eqref{eq:coro_proof_4}, inequality~\eqref{eq:coro_proof_2}, Lemma~\ref{lem:eta_mu} and Proposition~\ref{prop:gap_ratio}, we get
\begin{align*}
\eta_{x^K} &= \frac{y^\top (b_t - A_t x^K )}{3r + 1} \le (1+\beta_\theta) (1+C_{3r+1}) \eta(\mu_K) \le 2 (1+\beta_\theta) (1+C_{3r+1}) \gap (x^\phi(\mu_K))\\
&  \le \frac{2 (1+\beta_\theta) (1+C_{3r+1})}{1-\beta_\theta} \gap (x^K)\ \le 1,
\end{align*}
where $x^K \in \mathcal{N}^\phi_{\theta,t} (\mu_K)$. 
Similarly, using definition~\eqref{eq:coro_proof_4}, inequality~\eqref{eq:coro_proof_3}, Lemma~\ref{lem:eta_mu} and Proposition~\ref{prop:gap_ratio}, we get
\begin{align*}
\eta_{x^0} &= \frac{y^\top (b_t - A_t x^0 )}{3r + 1} \ge \frac{1-\beta_\theta}{1+C_\nu} \eta (\mu^0) \ge \frac{1-\beta_\theta}{(3r+1)(1+C_\nu)} \gap(x^\phi(\mu_0))\\
& \ge \frac{1-\beta_\theta}{(1+\beta_\theta)(3r+1)(1+C_\nu)} \gap(x^0) \ge \sqrt{t},
\end{align*}
where $x^0\in \mathcal{N}^\phi_{\theta, t} (\mu_0)$
Taking $\eta_0 = \eta_{x^0}$ and $\eta_K = \eta_{x^0}$, the proof is completed.
\end{proof}

\section*{Acknowledgments}
The authors thank Defeng Sun, Kim-Chuan Toh and Stephen Wright for their helpful discussions.

\bibliographystyle{abbrv}
\bibliography{references}
\end{document}